\documentclass[11pt,leqno]{amsart}
\usepackage{amsmath,amsthm,amssymb}
\usepackage{tabularx}

\newcommand{\C}{{\mathbb C}}
\newcommand{\CC}{{\mathbb C}}
\newcommand{\Z}{{\mathbb Z}}
\newcommand{\ZZ}{{\mathbb Z}}
\newcommand{\R}{{\mathbb R }}
\newcommand{\RR}{{\mathbb R}}

\newcommand{\NN}{{\mathbb N}}

\usepackage{color} 

\newtheorem{Thm}{Theorem}[section]
\newtheorem{thm}{Theorem}[section]

\newtheorem{lem}[Thm]{Lemma}

\newtheorem{prop}[Thm]{Proposition}

\theoremstyle{definition}

\newtheorem{df}[Thm]{Definition}

\newtheorem{rem}[Thm]{Remark}

\newtheorem{ex}[Thm]{Example}
\newtheorem{prob}[Thm]{Problem}

\makeatletter
\@addtoreset{equation}{section}

\makeatother

\title[ Pseudo-normalized Hecke eigenform]{Pseudo-normalized Hecke eigenform and its application to extremal $2$-modular lattices}

\author{Tsuyoshi Miezaki*}
\thanks{*Corresponding author}
\address{		Faculty of Science and Engineering, 
		Waseda University, 
		Tokyo 169--8555, Japan
}
\email{miezaki@waseda.jp} 

\author{Gabriele Nebe}
\address{RWTH Aachen University, 52056 Aachen, Germany}
\email{nebe@math.rwth-aachen.de}

\date{}



\begin{document}

\begin{abstract}
	It is shown that 
extremal $2$-modular lattices of 
ranks $32$ and $48$ are
generated by their vectors of minimal norm. 
In the proof, 
we use certain properties of the difference of 
normalized Hecke eigenforms. 
We refer to them as the pseudo-normalized Hecke eigenform, 
the concept of which is introduced 
in this paper.
\end{abstract}
\maketitle





\section{Introduction}

A lattice in $\R^{n}$ is a subset $L \subset \R^{n}$ 
containing a basis 
$\{e_{1}, \ldots, e_{n}\}$ of $\R^{n}$ 
such that $L =\Z e_{1}\oplus \cdots \oplus\Z e_{n}$, 
i.e., $L $ consists of all integral linear combinations 
of the vectors $e_{1}, \ldots, e_{n}$. 
The dual lattice $L$ is defined as
\[
L^{\sharp}:=\{y\in \R^{n}\mid (y,x) \in\Z , \ \forall x\in L\}, 
\]
where $(x,y)$ is the standard inner product. 
Herein, we assume that the lattice $L $ is integral, 
i.e., $(x,y) \in\Z$ for all $x$, $y\in L$. 
An integral lattice $L$ is even if $(x,x) \in 2\Z$ for all $x\in L$. 
An integral lattice $L$ is unimodular if $L^{\sharp}=L$. 

The notion of $\ell$-modular lattices is introduced in \cite{Q}. 
An $n$-dimensional integral lattice $L$ is modular if a similarity $\sigma$ of $\RR^n$ exists such that 
\[
\sigma(L^\sharp)=L, 
\]
where $L^\sharp$ is the dual lattice of $L$. 
If $\sigma$ multiplies norms by $\ell$, it is regarded as $\ell$-modular. 
For example, the root lattices $E_8$, $D_4$, $A_2$ are 
$1$-,$2$-,$3$-modular, respectively.
The $1$-modular lattices are better known as unimodular lattices.

Let $L$ be an even $2$-modular lattice of rank $n$. 
Then $n$ 
is divisible by $4$ and \cite{{Q},{RS}} give 
the following bound on the minimum norm of a nonzero vector in
$L$:
\begin{align}\label{ine:lattice-ext-2}
\min(L)\leq 2\left\lfloor\frac{n}{16}\right\rfloor+2. 
\end{align}
A $2$-modular lattice $L$ that achieves equality in 
(\ref{ine:lattice-ext-2}) is called extremal.

Herein, 
we investigate the following problem: 
\begin{prob}\label{prob:main}
Let $L$ be a lattice and 
\begin{align*}
L_{\ell_1,\ldots,\ell_s}:=\{x\in L\mid (x,x)=\ell_1 \mbox{ or } \cdots \mbox{ or } (x,x)=\ell_s\}.
\end{align*}
Is $L$ generated by 
$L_{\ell_1,\ldots,\ell_s}$?
\end{prob}
Let us introduce the known results of Problem \ref{prob:main}.
\begin{enumerate}
\item 
If $L$ is an extremal even unimodular lattice of rank 
$32$ or $48$, 
then $L$ is generated by its vectors of minimal norm \cite{{O1},{O2}}. 

\item 
If $L$ is an extremal even unimodular lattice of rank 
$56$, $72$ or $96$, 
then $L$ is generated by its vectors of minimal norm \cite{{K}}. 

\item 
If $L$ is an extremal
even unimodular lattice of rank $40$, 
then $L$ is generated by its vectors of norms $4$ and $6$ \cite{O3}. 

\item 
If $L$ is an extremal
even unimodular lattice of rank $80$ (resp.~$120$), 
then $L$ is generated by its vectors of norms $8$ and $10$ 
(resp.~norms $12$ and $14$) \cite{KA}. 

\end{enumerate}

The main result of this paper is 
 the following theorem: 
\begin{thm}\label{thm:main}
\begin{enumerate}

\item 
Let $L$ be an extremal even $2$-modular lattice of rank $32$, then 
$L$ is generated by its vectors of minimal norm.

\item 
Let $L$ be an extremal even $2$-modular lattice of rank $48$, then 
$L$ is generated by its vectors of minimal norm.

\item 
Let $L$ be an extremal even $2$-modular lattice of rank $24$, then 
$L$ is generated by 
its vectors of norms $4$ and $6$. 

\item 
Let $L$ be an extremal even $2$-modular lattice of rank $36$, then 
$L$ is generated by 
its vectors of norms $6$ and $8$. 
\end{enumerate}

\end{thm}

\begin{rem}
Let $L$ be an extremal 
even $2$-modular lattice of rank $n$ with $4\leq n\leq 20$. 
Then 
for $n\neq 12$, we can show that $L$ is generated 
by its vectors of minimal norm and 
for $n= 12$, we can show that $L$ is generated 
by its vectors of norms $2$ and $4$ 
with the same arguments of Theorem \ref{thm:main}. 
However, we omit to describe it here because 
these can be obtained by direct computations as all extremal even 
$2$-modular lattices in dimensions up to 20 are known 
\cite{{SV},{BV}}. 
\end{rem}

For the proof of Theorem \ref{thm:main}, 
we will introduce the notion of a pseudo-normalized Hecke eigenform.

This paper is organized as follows. In Section \ref{sec:pre}, 
we provide the definitions and basic properties of $2$-modular lattices, as well as the
spherical $t$-designs used in this study. 
In Section \ref{sec:proof}, 
the proof of Theorem \ref{thm:main} is provided, along with concluding remarks.

All computer calculations in this paper were done with the help of 
{\sc Magma} \cite{Magma} and {\sc Mathematica} \cite{Mathematica}. 

\section{Preliminaries}\label{sec:pre}

\subsection{Spherical $t$-designs}

The concept of a spherical $t$-design originated from Delsarte, Goethals, and Seidel 
\cite{DGS}. For a positive integer $t$, a finite nonempty set X in the unit sphere
\[
S^{n-1} = \{x = (x_1, \ldots , x_n) \in \R ^{n}\ |\ x_1^{2}+ \cdots + x_n^{2} = 1\}
\]
is known as a spherical $t$-design in $S^{n-1}$ if the following condition is satisfied:
\[
\frac{1}{|X|}\sum_{x\in X}f(x)=\frac{1}{|S^{n-1}|}\int_{S^{n-1}}f(x)d\sigma (x), 
\]
for all polynomials $f(x) = f(x_1, \ldots ,x_n)$ of degree not exceeding $t$. 
A finite subset $X$ in $S^{n-1}(r)$ for a sphere of radius $r$ 
centered at the origin 
is also called a spherical $t$-design if the appropriately rescaled 
set $(1/r)X$ is 
a spherical $t$-design on the unit sphere $S^{n-1}$.
Hence, we say that $L_\ell$ is a spherical $t$-design 
if $(1/\sqrt{\ell})L_\ell$ is a spherical $t$-design. 

Let $L$ be an extremal even $2$-modular lattice of dimension $n$, and 
let us set 
\[
L_\ell:=\{x\in L\mid (x,x)=\ell\}. 
\]
If the set is non-empty then 
$L_\ell$
forms a spherical $t$-design (\cite[Corollary 3.1]{BV}), where
\[
t=
\left\{
\begin{array}{ll}
7\ &{\rm if}\ n \equiv 0 \pmod{16},\\
5\ &{\rm if}\ n \equiv 4 \pmod{16},\\
3\ &{\rm if}\ n \equiv 8 \pmod{16}. 
\end{array}
\right.
\]

Let ${\rm {\rm Harm}}_{j}(\R^{n})$ denote the set of homogeneous 
harmonic polynomials of degree $j$ on $\R^{n}$. It is well known that $X$ is a spherical $t$-design if and only if the condition 
\[
\sum_{x\in X}P(x)=0 
\]
holds for all $P \in {\rm Harm}_j(\R^n)$ with $1 \leq  j \leq  t$. 
If the set $X$ is antipodal, i.e., $-X=X$, and $j$ is odd, then the aforementioned condition is fulfilled automatically. Hence, we can reformulate the condition of spherical $t$-design on an antipodal set as follows:
\begin{prop}\label{thm:design-lattice}
A nonempty finite antipodal subset $X\subset S^{n-1}_{m}$ is a spherical $2s+1$-design if the condition 
\[
\sum_{x\in X}P(x)=0 
\]
holds for all $P \in {\rm Harm}_{2j}(\R^n)$ with $2 \leq  2j \leq  2s$. 
\end{prop}

\subsection{Spherical theta series}                                   
Let $\mathbb{H} :=\{z\in\C\mid {\rm Im}(z) >0\}$ be the upper half-plane. 
\begin{df}
Let $L$ be the lattice of $\R^{n}$. Then, for a polynomial $P$, the function 
\[
\theta _{L, P} (z):=\sum_{x\in L}P(x)e^{i\pi z(x,x)}
\]
is known as the theta series of $L $ weighted by $P$. 
\end{df}

\begin{rem}
The weighted theta series can be written as 
\begin{equation}\label{weighted_th}
\theta _{L, P} (z)
=\theta _{L, P} (q)
=\sum_{m\geq 0}a^{(P)}_{m}q^{m}, 
\end{equation}
where $a^{(P)}_{m}:=\sum_{x\in L_{m}}P(x)$ and $q=e^{\pi i z}$. 
\end{rem}

For example, we consider an even $2$-modular lattice $L$. 
Subsequently, the weighted
theta series $\theta_{L,P}$ of $L$ weighted 
by a harmonic polynomial $P$ is of a modular 
form with respect to $\Gamma$. 
In general, we have the following:
\begin{prop}[\cite{{Q},{BV}}]\label{prop:BV}
Let $L$ be an even $2$-modular lattice $L$ of rank $n$ and 
$L':=\sqrt{\ell}L$. 
Then, for $P\in {\rm Harm}_{4k}(\RR^{n})$, we have
\[
\begin{cases}
\theta_{L, P}(z)+\theta_{L', P}(z)&\in \CC[\theta_{D_4}(q),\Delta_{16}(q)],\\ 
\theta_{L, P}(z)-\theta_{L', P}(z)&\in \Phi_{24}(q)\CC[\theta_{D_4}(q),\Delta_{16}(q)], 
\end{cases}
\]
for $P\in {\rm Harm}_{4k+2}(\RR^{n})$, 
\[
\begin{cases}
\theta_{L, P}(z)+\theta_{L', P}(z)&\in \Phi_{24}(q)\CC[\theta_{D_4}(q),\Delta_{16}(q)],\\ 
\theta_{L, P}(z)-\theta_{L', P}(z)&\in \CC[\theta_{D_4}(q),\Delta_{16}(q)], 
\end{cases}
\]
where 
\[
\Delta_{16}(q)=(\eta(q)\eta(2q))^8=q^2+\cdots, 
\]
$\eta(q)$ is the Dedekind $\eta$-function, and 
$\Phi_{24}(q)=q^2+\cdots$ is a modular form of weight $12$. 
For more details, see \cite{{Q},{BV}}. 
\end{prop}

Using Proposition \ref{prop:BV}, 
for $P\in {\rm Harm}_{8}(\RR^{16m})$ and $P\in {\rm Harm}_{10}(\RR^{16m})$, 
$\theta_{L, P}(z)$ can be written explicitly as follows: 
\begin{lem}\label{lem:main-2}
\begin{enumerate}
\item [{\rm (1)}]
Let 
\[
d(r)=
\begin{cases}
8\ (\mbox{if }r=0),\\
6\ (\mbox{if }r=4),\\
4\ (\mbox{if }r=8).
\end{cases}
\]
Let $L$ be an extremal even $2$-modular 
lattice of rank $16m+r\ (r=0,4,8)$ and 
$P\in {\rm Harm}_{d(r)}(\RR^{16m+r})$. 
Subsequently, we have 
\[
\begin{cases}
\theta_{L',P} = \theta_{L, P}(z)=c_1\Delta_{16}(q)^{m+1}
=c_1(q^{2m+2}-{8(m+1)}q^{{2m+4}}+\cdots) \mbox{ if } r=0,8, \\
-\theta_{L',P} = \theta_{L, P}(z)=c_1\Delta_{16}(q)^{m+1}
=c_1(q^{2m+2}-{8(m+1)}q^{{2m+4}}+\cdots) \mbox{ if } r=4. 
\end{cases}
\]
\item [{\rm (2)}] 
Let $L$ be an extremal even $2$-modular 
lattice of rank $16m+r\ (r=0,4,8)$ and 
$P\in {\rm Harm}_{d(r)+2}(\RR^{16m+r})$. 
Subsequently, we have 
\[
\begin{cases}
-\theta_{L',P} = \theta_{L,P}=c_2\theta_{D_4}(z)\Delta_{16}(z)^{m+1}
=c_2(q^{2m+2}-{8(m-2)}q^{{2m+4}}+\cdots) \mbox{ if } r=0,8, \\
\theta_{L',P} = \theta_{L,P}=c_2\theta_{D_4}(z)\Delta_{16}(z)^{m+1}
=c_2(q^{2m+2}-{8(m-2)}q^{{2m+4}}+\cdots) \mbox{ if } r=4.
\end{cases}
\]
Moreover, for $m=2$ we have 
\begin{align*}
\theta_{L,P}&=c_2\theta_{D_4}(z)\Delta_{16}(z)^{3}
=c_2(q^{6} - 324 q^{10} + 4096 q^{12}+\cdots)\\
&=c_2\sum_{m=1}^{\infty}a(m)q^m\ (\mbox{say}) 
\end{align*}
and  
\begin{enumerate}
\item [{\rm (i)}]
$a(2^i)=0$ for all $i\in \NN$, 
\item [{\rm (ii)}]
$a(2^i3)=2^{12i}$ for all $i\in \NN$. 
\end{enumerate}
\end{enumerate}
\end{lem}

\begin{proof}
(1)
We prove the statement for $n=16m$ case only, 
the other cases can be proved in the same way. 
Let $L$ be an extremal even $2$-modular lattice of rank $16m$ and 
$P\in {\rm Harm}_{8}(\RR^{16m})$. 
Therefore, $\theta_{L, P}(z)$ and $\theta_{L', P}(z)$ 
are modular forms of weight $8m+8$. 
We remark that 
by the extremality, the leading term is $c q^{2m+2}+\cdots$ 
for some constant $c\in \RR$. 
By Proposition \ref{prop:BV}, for some constant $c\in \RR$, we have
\begin{align*}
\theta_{L, P}(z)+\theta_{L', P}(z)&=c\Delta_{16}(q)^{m+1},\\ 
\theta_{L, P}(z)-\theta_{L', P}(z)&=0.
\end{align*}
Subsequently,  
\[
	\theta _{L',P} = \theta_{L, P}(z)=c_1\Delta_{16}(q)^{m+1}.
\]
(2)
We prove the statement for $n=16m$ case only, 
the other cases can be proved in the same way. 
Let $L$ be an extremal even $2$-modular lattice of rank $16m$ and 
$P\in {\rm Harm}_{10}(\RR^{16m})$. 
Therefore, $\theta_{L, P}(z)$ and $\theta_{L', P}(z)$ 
are modular forms of weight $8m+10$. 
We remark that 
by the extremality, the leading term is $c q^{2m+2}+\cdots$ 
for some constant $c\in \RR$. 
By Proposition \ref{prop:BV}, for some constant $c\in \RR$, we have
\begin{align*}
\theta_{L, P}(z)+\theta_{L', P}(z)&=0,\\ 
\theta_{L, P}(z)-\theta_{L', P}(z)&=c\theta_{D_4}\Delta_{16}(q)^{m+1}.
\end{align*}
Subsequently, 
\[
\theta_{L, P}(z)=c_2\theta_{D_4}\Delta_{16}(q)^{m+1}. 
\]

We remark that $\theta_{D_4},\Delta_{16}(q)$ are 
modular forms for $\Gamma_0(2)$ \cite{Kob}. 
The dimension of the space of cusp forms of weight $26$ 
for $\Gamma_0(2)$ is five; using {\sc Magma} \cite{Magma},
we obtain the basis as follows: 
\begin{align*}
f_1&=q^2 + 2657760q^{12} - 21963256q^{14} + 1015627776q^{16} - 8615579463q^{18} +\cdots,\\
f_2&=  q^4 - 252252q^{12} - 1032192q^{14} - 42991616q^{16} - 54853632q^{18} -
\cdots,\\
f_3&=  q^6 + 19648q^{12} + 256770q^{14} + 2654208q^{16} + 16097088q^{18} +\cdots,\\
f_4&=  q^8 - 1176q^{12} - 21504q^{14} - 196656q^{16} - 1142784q^{18} -\cdots,\\
f_5&=  q^{10} + 48q^{12} + 852q^{14} + 8192q^{16} + 48510q^{18} +\cdots.
\end{align*}
Subsequently, the following are normalized Hecke eigenforms: 
\begin{align*}
h_1&=f_1 + 4096f_2 + 12 \left(15827+400 \sqrt{106705}\right)f_3 \\
&+ 16777216f_4 + 150 \left(2473177-10368 \sqrt{106705}\right)f_5\\
&=q^2+4096 q^4+12 \left(15827-400 \sqrt{106705}\right) q^6+16777216 q^8\\
&+150 \left(2473177+10368 \sqrt{106705}\right) q^{10}+\cdots,\\
h_2&=f_1 + 4096f_2 + 12 \left(15827-400 \sqrt{106705}\right)f_3 \\
&+ 16777216f_4 + 150 \left(2473177+10368 \sqrt{106705}\right)f_5\\
&=q^2+4096 q^4+12 \left(15827+400 \sqrt{106705}\right) q^6+16777216 q^8\\
&+150 \left(2473177-10368 \sqrt{106705}\right) q^{10}+\cdots.  
\end{align*}
By comparing the Fourier coefficients, 
we have 
\[
\theta_{D_4}(z)\Delta_{16}(z)^3=\frac{h_2-h_1}{9600\sqrt{106705}}. 
\]
For $j=1,2$, we denote by $c_{h_j}(n)$ the coefficient of $h_j$ as follows:
\[
h_j =\sum_{n=1}^{\infty}c_{h_j}(n)q^n.
\]
Because $h_1$ and $h_2$ are normalized Hecke eigenforms, 
the Fourier coefficients of $h_j$ satisfy the following equations{\rm :}
\begin{align}
c_{h_j}(mn)&=c_{h_j}(m)c_{h_j}(n)\ (m, n \ {\rm coprime}) \label{eqn:mul},\\
c_{h_j}(p^{\alpha +1})&=c_{h_j}(p)c_{h_j}(p^{\alpha })-p^{k-1}c_{h_j}(p^{\alpha -1}) \ (p \mbox{ is a prime with }p\neq 2) \label{eqn:rec}, \\
c_{h_j}(2^{\alpha +1})&=c_{h_j}(2)c_{h_j}(2^{\alpha })\label{eqn:rec2}. 
\end{align}
Because $c_{h_1}(2)=c_{h_2}(2)=4096$ and by applying (\ref{eqn:rec}), 
we obtain (i). 

Using (\ref{eqn:rec2}), we obtain 
\[
c_{h_j}(2^{i})=c_{h_j}(2)c_{h_j}(2^{i-1})=c_{h_j}(2)^i=4096^i. 
\]
Using (\ref{eqn:mul}), for $j=1,2$, $c_{h_j}(2^i3)=c_{h_j}(2^i)c_{h_j}(3)$. 
Subsequently, 
\begin{align*}
c_{h_2}(2^i3)-c_{h_1}(2^i3)
&=c_{h_2}(2^i)c_{h_2}(3)-c_{h_1}(2^i)c_{h_1}(3)\\
&=c_{h_2}(2^i)(9600 \sqrt{106705})
\end{align*}
Namely, 
\begin{align*}
a(2^i3)=4096^i=2^{12i}. 
\end{align*}
The proof of (2) is completed. 
\end{proof}

Next, we define the concept of a pseudo-normalized Hecke eigenform. 
\begin{df}
Let $f$ be a pseudo-normalized Hecke eigenform of weight $k$ for some group $\Gamma$
if 
\[
f=g_1-g_2,
\]
where
$g_1$ and $g_2$ are normalized Hecke eigenforms of weight $k$ 
for $\Gamma$. 
\end{df}
\begin{ex}
By Lemma \ref{lem:main-2}, 
$\theta_{D_4}\Delta_{16}^3$ is a pseudo-normalized Hecke eigenform.
\end{ex}

\section{Proof of Theorem \ref{thm:main}}\label{sec:proof}
Let $\mathcal{L}(L_{m_1,\ldots,m_k})$ be the lattice generated by 
$L_{m_1},\ldots, L_{m_k}$. Let $m_0:=\min(L)$.  
For the proof, 
suppose an equivalence class $[x']\in L/\mathcal{L}(L_{m_0})$ exists
such that $x'$ is a minimal-norm representative with norm $(x',x')=s>m_0$. 
For $j \in \ZZ$, we write
\begin{align*}
	M_j:=M_j(L;x') &:=|\{
x\in L_{\min(L)}\mid |(x,x')| =j
\}|,\\
	M'_j:=M'_j(L;x') &:=|\{
		x\in L^{\sharp }_{\min(L^{\sharp })}\mid |(x,x')| =j
\}|,\\
N_j:=N_j(L;x') &:=|\{
x\in L_{\min(L)+2}\mid |(x,x')|=j
\}|.
\end{align*}
Then we have the following results: 

\begin{lem}[\cite{{Venkov2},{KA}}]\label{lem:main-3}
\begin{enumerate}
\item [{\rm (1)}]
For all $x \in L_{m_0}$, we have the inequality
\[
|(x,x')|\leq \frac{m_0}{2}. 
\]
\item [{\rm (2)}]
For all $x \in L_{m_0+2}$, we have the inequality
\[
|(x,x')|\leq \frac{m_0}{2}+1. 
\]
\end{enumerate}
\end{lem}
\begin{lem}[\cite{{Venkov2},{KA}}]\label{lem:main} 
Let 
\[
\theta_{L}(q)=\sum_{i=0}^{\infty}a_{i}(L)q^i. 
\]
\begin{enumerate}
\item [{\rm (1)}]We have 
\begin{align*}
a_{m_0}(L) &= 
\sum_{j=0}^{m_0/2}M_j(L;x'),\\
a_{m_0+2}(L) &= 
\sum_{j=0}^{m_0/2+1}N_j(L;x'). 
\end{align*}
\item [{\rm (2)}]
Let $L$ be a lattice such that 
		for $m\in \{m_0,m_0+2\}$, $L_m$ is a spherical $(2t+1)$-design. Then we have that
\begin{align*}
\sum_{j=1}^{m_0/2}
j^{2k}M_j(L; x')
&=a_{m_0}(L)
\frac{1\cdot3\cdots (2k-1)}{n\cdot(n+2)\cdots (n+2k-2)}m_0^k(x',x')^k, \\
\sum_{j=1}^{m_0/2+1}
j^{2k}N_j(L; x')
&=a_{m_0+2}(L)
\frac{1\cdot3\cdots (2k-1)}{n\cdot(n+2)\cdots (n+2k-2)}(m_0+2)^k(x',x')^k, 
\end{align*}
for $k\in\{1,\ldots,t\}$. 
\item [{\rm (3)}] 
	If $L_{m_0}$ is a spherical $2$-design, then 
		$(x',x') \leq \frac{n \cdot m_0}{4} $. 
\end{enumerate}
\end{lem}
Note that part (3) gives a general upper bound on the minimal norm 
in a class of $L/{\mathcal L}(L_{m_0})$ for 2-design lattices. 
In any concrete case (see for example below) we find much better
upper bounds. 

To conclude part (3) we read part (2) for $k=1$ and note that 
the left hand side is $\leq (m_0/2)^2 a_{m_0}(L) $. 
So $a_{m_0}(L) m_0 (x',x') /n \leq a_{m_0}(L) m_0^2 / 4$, whence 
$(x',x') \leq n m_0 /4 $. 
\\

Next, we present a proof of Theorem \ref{thm:main}. 
\begin{proof}[Proof of Theorem \ref{thm:main} (1)]

Let $L$ be an extremal even $2$-modular lattice of rank $32$. 
Let $m_0:=\min(L)=6$ and $t(32):=3$. 

We show 
that  any class 
$[x] \in L/\mathcal{L}(L_{m_0})$, 
$[x]$ is represented by a vector $x'\in [x]$ 
with norm $(x',x')\leq m_0$. 

Suppose that an equivalence class 
$[x'] \in L/\mathcal{L}(L_{m_0})$ exists such that 
$[x']$ is a minimal-norm representative 
with norm $(x',x')= s > m_0$.

	Then  by Lemma \ref{lem:main-3} we have 
	$|(x',v)| \leq  3 = m_0/2 $ for all $v\in L_{6}$. 
For the values $M_i:=M_i(L,x') := |\{ v\in L_6 \mid |(x',v) | = i \} |$ 
	from Lemma \ref{lem:main} we have that $M_i\neq 0$ only 
	if $i=0,1,2,3$ and 
	$$\begin{array}{l} 
		M_0+M_1+M_2+M_3 = a_6(L) = | L_6 | = 261120, \\
		\sum _{v\in L_6} (x',v)^2 = \sum _{i=1}^3 i^2 M_i = 6 s a_6(L)/32,  \\
		\sum _{v\in L_6} (x',v)^4 = \sum _{i=1}^3 i^4 M_i = 6^2 s^2 3a_6(L)/(32\cdot 34),  \\
		\sum _{v\in L_6} (x',v)^6 = \sum _{i=1}^3 i^6 M_i = 6^3 s^3 3\cdot 5 a_6(L)/(32\cdot 34 \cdot 36). 
	\end{array} $$
	Regarding $s$ as a parameter we have 4 equations with 4 unknowns 
	having a unique solution: 
	$$\begin{array}{l} 
		M_0 = -600s^3 + 10080s^2 - 66640s + 261120, \\
		M_1 = 900s^3 - 14040s^2 + 73440s, \\
		M_2 =  -360s^3 + 4320s^2 - 7344s, \\
		M_3 = 60s^3 - 360s^2 + 544s. 
	\end{array} $$
	The polynomial $M_2/s$ is only positive when 
	$s \in [2.05,9.94] $. As $s=(x',x') $ is an even positive integer
	$>m_0 = 6$ 
	we conclude that $s=8$ is the only possible solution. 
	Putting $s=8$ we obtain 
	$$M_0 = 65920,\ M_1=149760, \ 
M_2=33408, \
M_3=12032 .$$

Now we need the harmonic polynomial $P_{8,x'}$ of degree $8$:
\begin{align*}
P_{8,x'}(x)=&(x,x')^8- \frac{7}{11} (x,x')^6(x',x')(x,x)
+ \frac{5}{44}(x,x')^4(x',x')^2(x,x)^2 \\
&- \frac{1}{176}(x,x')^2(x',x')^3(x,x)^3+\frac{1}{26752}(x',x')^4(x,x)^4.   
\end{align*}
By Lemma \ref{lem:main-2} (1),
\[
\sum_{v\in L_{6}}
P_{8,x'}(v) = 
\sum_{w\in L'_{6}}
P_{8,x'}(w). 
\]

	Recall that $L' = \sqrt{2} L^{\sharp }$. 
	For $w\in L^{\sharp} _{3}$ and $x' \in L_8$ we have 
	$$|(w,x')| \in \{ 0 ,1, 2,3 \} $$ 
	because if $(w,x') = 4$ then 
	$x'-2w \in L$ is a vector of norm 4 in $L$ contradicting the
	extremality of $L$. 
	Also $L^{\sharp }_3$ is a spherical $6$-design, so putting
	$M'_j := | \{ w\in L^{\sharp }_3  \mid |(x',w) | = j \} |$ 
	we find with the equations in Lemma \ref{lem:main} that 
	$$M'_0 = 117440 , \  M'_1=126720, \
M'_2= 16704 , \ 
M'_3=   256 . $$
From these numbers we  compute 
$$\sum_{w\in L'_{6}}
P_{8,x'}(w) = -8847360/19 $$
which implies that 
$$\sum _{v\in L_6} (x',v) ^8 = 97320960 = 2^{16} 3^3 5 \cdot 11 .$$
This equation is not satisfied by the $M_i$ above. 
\end{proof}

\begin{rem}
	From our calculations we obtain that for $x'\in L_8$ the 
	unique solution for $M_i := | \{ v\in L_6 \mid |(x',v)| = i \} |$ 
	is 
	$$M_0 = 82720, \ M_1=  122880,\ M_2= 46848,\  M_3= 8192,\  M_4= 480 .$$
\end{rem}

\begin{proof}[Proof of Theorem \ref{thm:main} (2)]

Let $L$ be an extremal even $2$-modular lattice of rank $48$. 
Let $m_0:=\min(L)=8$ and $t(48):=3$. 

We show 
that  any class 
$[x] \in L/\mathcal{L}(L_{m_0})$, 
$[x]$ is represented by a vector $x'\in [x]$ 
with norm $(x',x')\leq m_0$. 

Suppose that an equivalence class 
$[x'] \in L/\mathcal{L}(L_{m_0})$ exists such that 
$[x']$ is a minimal-norm representative 
with norm $(x',x')= s > m_0$. 

	{\bf Claim 1:} {$s\leq 18$} \\
First, we show that $s\leq 18$. 
By Lemma \ref{lem:main-3} we have 
	$|(x',v)| \leq  4 = m_0/2 $ for all $v\in L_{8}$.
For the values  $M_i:=M_i(L,x') := |\{ v\in L_{8} \mid |(x',v) | = i \} |$ 
	from Lemma \ref{lem:main} (2), we have that $M_i\neq 0$ only 
	if $i=0,1,2,3,4$.
	Since $L_8$ is a 6-design we obtain 
	$$\begin{array}{l} 
		M_0+M_1+M_2+M_3+M_4 = a_{8}(L) = | L_{8} | = 9828000, \\
		\sum _{v\in L_{8}} (x',v)^2 = \sum _{i=1}^4 i^2 M_i 
= 8 s a_{8}(L)/48,  \\
		\sum _{v\in L_{8}} (x',v)^4 = \sum _{i=1}^4 i^4 M_i 
= 8^2 s^2 3a_{8}(L)/(48\cdot 50),  \\
		\sum _{v\in L_{8}} (x',v)^6 = \sum _{i=1}^4 i^6 M_i 
= 8^3 s^3 3\cdot 5 a_{8}(L)/(48\cdot 50 \cdot 52). 
	\end{array} $$
Then we have: 
	$$\begin{array}{l} 
M_0 =  -16800s^3 + 305760s^2 - 2229500s + 35M_4 + 9828000, \\
M_1 = 25200s^3 - 425880s^2 + 2457000s - 56M_4, \\
M_2 =  -10080s^3 + 131040s^2 - 245700s + 28M_4, \\
M_3 = 1680s^3 - 10920s^2 + 18200s - 8M_4. 
	\end{array} $$
Since $M_2\geq 0$ and $M_3\geq 0$, we have 
\begin{align*}
360s^3 - 4680s^2 + 8775s
\leq 
M_4
\leq
210s^3 - 1365s^2 + 2275s.
\end{align*}
	As $s\geq 0$ this implies that $-s(s-221/10)\geq 130/3 $ 
	showing that $s\leq 18$.

	{\bf The harmonic polynomials:} \\
Now we need the harmonic polynomials $P_{8,x'}$ of degree $8$ and 
$P_{10,x'}$ of degree $10$:
\begin{align*}
P_{8,x'}(x)=&(x,x')^8- \frac{7}{15} (x,x')^6(x',x')(x,x)
+ \frac{7}{116}(x,x')^4(x',x')^2(x,x)^2 \\
&- \frac{1}{464}(x,x')^2(x',x')^3(x,x)^3+\frac{1}{100224}(x',x')^4(x,x)^4,   \\
P_{10,x'}(x)=&(x,x')^{10}- 
 \frac{45}{64} (x,x')^8(x',x')(x,x)  + \frac{315}{1984}(x,x')^6(x',x')^2(x,x)^2\\
&- \frac{105}{7936}(x,x')^4(x',x')^3(x,x)^3  
+ \frac{315}{902576} (x,x')^2 (x',x')^4(x,x)^4\\
&-\frac{9}{7364608}(x',x')^5(x,x)^5.
\end{align*}
	For the rescaled dual lattice $L' = \sqrt{2} L^{\sharp }$ we obtain
by Lemma \ref{lem:main-2} (1) and (2),
\begin{align}\label{eqn:48d}
\left\{
\begin{array}{l}
\displaystyle
\sum_{v\in L_{8}}
P_{8,x'}(v) = 
\sum_{w\in L'_{8}}
P_{8,x'}(w), \\
\displaystyle
\sum_{v\in L_{8}}
P_{10,x'}(v) = 
-\sum_{w\in L'_{8}}
P_{10,x'}(w). 
\end{array}
\right.
\end{align}

{\bf Claim 2:} {${\mathcal L}(L_8) = {\mathcal L}(L_{8,10,12}) $.}  \\
So assume that $s\leq 12$. 
Then $s=10$ or $s=12$ and 
for $v\in L^{\sharp }_4$ we have
$|(v,x')| \leq 5$. Otherwise there is $v\in L^{\sharp }_4$ with
$(v,x') \geq 6$ yielding a vector $2v-x'\in L$ of norm
$$(2v-x' , 2v-x' ) = 16-24+s = s-8 \leq 4 .$$
	Also $L^{\sharp }_4$ is a spherical $6$-design, so putting
	$M'_j := | \{ w\in L^{\sharp }_4  \mid |(x',w) | = j \} |$ 
a system of $10$ equations 
in the $12$ variables $s,M_j\ (0 \leq j \leq  4),M'_j\ (0 \leq j \leq  5)$ is obtained, 
implying that 
$$ 
f(s) + (32-4s)M'_4+(384-40s)M'_5 = 0 $$  where  
$$f(s) =  6s^5 - 210s^4 + 2681s^3 - 14742s^2+29120s  .$$
For $s=10$ and $s=12$ the coefficients $32-4s$ and $384-40s$ are
negative and so is $f(10)$ and $f(12)$ yielding 
the desired contradiction.

{\bf Claim 3}: {${\mathcal L}(L_8)$ contains all elements $2v$ with $v\in L^{\sharp}_4$.} \\
Let $v\in L^{\sharp }_4$. 
Then $2v \in L$ has norm 16. We show that there is a vector $x\in L_8$ 
such that either $(x,v) = 4$ or $(x,v) = 3$.
In the first case $2v-x \in L_8$ and in the second case $2v-x\in L_{12}$,
so by Claim 2 we obtain $2v\in {\mathcal L}(L_8)$. 
So in the notation above we need 
to show that $M_3'(L,x) \neq 0$ or  $M_4'(L,x)\neq 0$.
Assume that $M_4'(L,x) = 0$.
Then the fact that $L^{\sharp }_4$ is a 6-design 
yields 4 equations 
on the 4 unknowns $M_j'(L,x)$ for $j=0,1,2,3$ having the 
unique solution 
$M_0' = 4726960, M_1' = 4626720, M_2'=  468720, M_3' =    5600$. 
In particular $M_3' > 0$.

{\bf Claim 4}: {${\mathcal L}(L_8) $ contains $L_{14}$, $L_{16}$ and $L_{18}$.} \\ 
So let $x' \in L$ be a minimal representative of $x' + {\mathcal L}(L_8)$ 
of norm $s=(x',x')$. By the above we know that $s=14,16,$ or $18$. 
Also $M_j\neq 0$ only for $j=0,1,2,3,4$ by Lemma \ref{lem:main-3}. 
By Claim 3 we have $(x'-2v,x'-2v) \geq s$ for all $v\in L^{\sharp}_4$, so 
$M'_j \neq 0 $ only for $j=0,1,2,3,4$. 
We have 4 equations for the $M_j$ because $L_8$ is a 6-design, 
4 equations for the $M'_j$ because $L'_8$ is a 6-design, as well 
as the two equations from Lemma \ref{lem:main-2} 
 which admit a 
unique solution in all three cases: 
$$\begin{array}{|c|r|r|r|r|r|} 
	\hline
	s  & M_0 & M_1 & M_2 & M_3 & M_4  \\
	\hline
	14 & 
	12872510/3 & 3361792/3 & 12184144/3 & 50176/3 & 1015378/3 \\
	\hline 
	16 & 7466480 & -4652032 &  7406336 & -1074176 & 681392 \\ 
	\hline 
	18 & 
	13864158 & -78591744/5 & 68553072/5 & -16296192/5 & 6154074/5 \\
	\hline
\end{array} $$
This is clearly a contradiction as all $M_j$ are non-negative integers. 
\end{proof}

\begin{proof}[Proof of Theorem \ref{thm:main} (3)]

Let $L$ be an extremal even $2$-modular lattice of rank $24$. 
Let $m_0:=\min(L)=4$ and $t(24):=1$. 

We show that for any class 
$[x] \in L/\mathcal{L}(L_{m_0,m_0+2})$, 
$[x]$ is represented by a vector $x'\in [x]$ 
with norm $(x',x')\leq m_0+2$. 

Suppose that an equivalence class 
$[x'] \in L/\mathcal{L}(L_{m_0,m_0+2})$ exists such that 
$[x']$ is a minimal-norm representative 
with norm $(x',x')= s > m_0+2=6$. 

By Lemma \ref{lem:main}, 
a system of $2(t(24) + 1)=4$ equations 
in the
variables $s, M_j(L; x' )\ (0 \leq j \leq  2)$ 
and $N_j(L; x' )\ (0 \leq j \leq  3)$ is provided.

Now we need the harmonic polynomials $P_{4,x'}$ of degree $4$ and 
$P_{6,x'}$ of degree $6$:
\begin{align*}
P_{4,x'}(x)=&(x,x')^4- \frac{3}{14} (x,x')^2(x',x')(x,x)
+ \frac{3}{728}(x',x')^2(x,x)^2, \\
P_{6,x'}(x)=&(x,x')^{6}- 
 \frac{15}{32} (x,x')^4(x',x')(x,x)  + \frac{3}{64}(x,x')^2(x',x')^2(x,x)^2\\
&- \frac{1}{1792}(x',x')^3(x,x)^3.
\end{align*}
By Lemma \ref{lem:main-2} (1),
\[
\sum_{x\in L_{6}}
P_{4,x'}(x) = 
-16
\sum_{x\in L_{4}}
P_{4,x'}(x). 
\]
By Lemma \ref{lem:main-2} (2),
\[
\sum_{x\in L_{6}}
P_{6,x'}(x) = 
8
\sum_{x\in L_{4}}
P_{6,x'}(x). 
\]

By Lemmas \ref{lem:main} and \ref{lem:main-2}, 
a system of $6$ equations in the $8$ variables $s,M_j\ (0 \leq j \leq  2),N_j\ (0 \leq j \leq  3)$ is obtained and we have 
\begin{align*}
M_1 &=-336 (-4 + s) s + (2 N_3)/(-2 + s),\\
N_0 &=8064 (32 + (-9 + s) s) + (2 (22 - 5 s) N_3)/(-2 + s). 
\end{align*}
Since $M_1\geq 0$ and $N_0\geq 0$, we have 
\begin{align*}
1344 s - 1008 s^2 + 168 s^3
\leq 
N_3
\leq
\frac{-258048 + 201600 s - 44352 s^2 + 4032 s^3}{-22 + 5 s}. 
\end{align*}
As $s>0$ this implies that $s\leq 8$, in particular 
for $s=8$ we have 
$M_0=2016, 
M_1 = 0, 
M_2 = 1008, 
N_0 = 0, 
N_1= 225792, 
N_2 = 0, 
N_3=32256$.

So $(x',v)$ is even for all $v\in L_4$ and $(x',w)$ is odd for all 
$w\in L_6$. 
This implies that no vector of norm 6 in $L$ is the sum of two vectors
of norm 4. Therefore the inner products of all vectors $v_1,v_2\in L_4$ 
lie in $\{ 0 , \pm 2 , \pm 4 \}$. 
So $\sqrt{2}^{-1} L_4$ is a root system in 24-dimensional space consisting 
of 3024 roots, which is impossible, by the classification of irreducible root 
systems. 

\end{proof}

\begin{proof}[Proof of Theorem \ref{thm:main} (4)]

Let $L$ be an extremal even $2$-modular lattice of rank $36$. 
Let $m_0:=\min(L)=6$ and $t(36):=2$. 

We show that for any class 
$[x] \in L/\mathcal{L}(L_{m_0,m_0+2})$, 
$[x]$ is represented by a vector $x'\in [x]$ 
with norm $(x',x')\leq m_0+2$. 

Suppose that an equivalence class 
$[x'] \in L/\mathcal{L}(L_{m_0,m_0+2})$ exists such that 
$[x']$ is a minimal-norm representative 
with norm $(x',x')= s > m_0+2$.

First, we show that $s\leq 10$. 
By Lemma \ref{lem:main}, 
a system of $2(t(36) + 1)=6$ equations 
in the
variables $s, M_j(L; x' )\ (0 \leq j \leq  3)$ 
and $N_j(L; x' )\ (0 \leq j \leq  4)$ is provided.

Now we need the harmonic polynomials $P_{6,x'}$ of degree $6$ and 
$P_{8,x'}$ of degree $8$:
\begin{align*}
P_{6,x'}(x)=&(x,x')^6- \frac{15}{44} (x,x')^4(x',x')(x,x)
+ \frac{15}{616}(x,x')^2(x',x')^2(x,x)^2\\
&-\frac{1}{4928}(x',x')^3(x,x)^3, \\
P_{8,x'}(x)=&(x,x')^{8}- 
 \frac{7}{12} (x,x')^6(x',x')(x,x)  + \frac{35}{368}(x,x')^4(x',x')^2(x,x)^2\\
&- \frac{35}{8096}(x,x')^2(x',x')^3(x,x)^3+ \frac{5}{194304}(x',x')^4(x,x)^4.
\end{align*}
By Lemma \ref{lem:main-2} (1),
\[
\sum_{x\in L_{8}}
P_{6,x'}(x) = 
-24
\sum_{x\in L_{6}}
P_{6,x'}(x). 
\]
By Lemma \ref{lem:main-2} (2),
\[
\sum_{x\in L_{8}}
P_{8,x'}(x) = 
0. 
\]

By Lemmas \ref{lem:main} and \ref{lem:main-2}, 
a system of $8$ equations in the $10$ variables $s,M_j\ (0 \leq j \leq  3),N_j\ (0 \leq j \leq  4)$ is obtained and we have 
\begin{align*}
M_2 &=-9 s (1425 - 750 s + 86 s^2) + (3 N_4)/(-3 + s),\\
N_3 &=1020 s (19 + 6 (-4 + s) s) + (4 (9 - 2 s) N_4)/(-3 + s). 
\end{align*}
Since $M_2\geq 0$ and $N_3\geq 0$, we have 
\begin{align*}
&-12825 s + 11025 s^2 - 3024 s^3 + 258 s^4\leq \\
&N_3
\leq
(-14535 s + 23205 s^2 - 10710 s^3 + 1530 s^4)/(-9 + 2 s). 
\end{align*}
As $s>0$ this implies that $s\leq 10$.

Finally, we show that $s\leq 8$. 
Let $x' \in L$ be a minimal representative of $x' + {\mathcal L}(L_{6,8})$ 
of norm $s=(x',x')$. 
Assume that $s=10$. 
Also $M_j\neq 0$ only for $j=0,1,2,3$ by Lemma \ref{lem:main-3}. 
For $v\in L^{\sharp}_3$ 
we have $2v\in L_{12}\subset{\mathcal L}(L_{6,8})$ and 
$(x'-2v,x'-2v) \geq s$ for all $v\in L^{\sharp}_3$, so 
$M'_j \neq 0 $ only for $j=0,1,2,3$. 
We have $3$ equations for the $M_j$ because $L_8$ is a $4$-design, 
$3$ equations for the $M'_j$ because $L'_6$ is a $4$-design, as well 
as the two equations from Lemma \ref{lem:main-2} 
which admit a unique solution: 
\[
M'_3=575. 
\]
This is a contradiction since $M'_3\in 2\ZZ$. 
\end{proof}

\begin{rem} 
	Michael J\"urgens used similar methods in his thesis 
	to prove that an extremal 3-modular lattice of dimension 36 
	is generated by its minimal vectors (\cite[Satz 2.6.2]{Jur}). 
	We checked the results of Theorem \ref{thm:main} 
	with a slight modification of J\"urgens' program, by which we also
	found unique solutions for the configuration numbers
	in the case of dimension 24 and 36. 
	\\
	If $L$ is an extremal 2-modular lattice of dimension 24 and 
	$x' \in L \setminus {\mathcal L}(L_4)$ a vector of norm 6, 
	then 
	$$
	\begin{array}{llll} 
		M_0=1116, &  M_1= 1536,  & M_2=372, &  N_0=83052 ,  \\
		N_1=119040,&  N_2= 47646,&   N_3=  7936 ,&  N_4 = 372  \\
		M'_0 =   1602,  & M'_1=   1392, & \mbox{ and } & M'_2 =     30. \\
	\end{array} $$
	If $L$ is an extremal 2-modular lattice of dimension 36 and 
	$x' \in L \setminus {\mathcal L}(L_6)$ a vector of norm 8, 
	then 
	$$
	\begin{array}{lllll} 
		M_0 = 56320 , & M_1 = 77760, & M_2 =  25920, & M_3 =  4160, & 
		 N_0 = 6416070 \\ N_1 = 9953920 , & N_2 = 4439040 , & N_3 = 1051968, &  
		 N_4 = 111760, & N_5 =  4160\\   M'_0 = 77840 , & M'_1 = 78840 , &  M'_2 = 7344 , & \mbox{ and } & M'_3 = 136.  \\
	\end{array} 
	$$
\end{rem}

\section*{Acknowledgments}

The authors would also like to thank the anonymous reviewers 
for their beneficial comments on an earlier version of the manuscript. 
The first named author is supported by JSPS KAKENHI (22K03277).



\begin{thebibliography}{999}







\bibitem{BV}
C.~Bachoc and B.~Venkov, Modular forms, lattices and spherical designs, R\'eseaux euclidiens, designs sph\'eriques et formes modulaires, 87--111, 
Monogr.~Enseign.~Math., 37, Enseignement Math., Geneva, 2001.




\bibitem{Magma}
W.~Bosma, J.~Cannon, and C.~Playoust.
The Magma algebra system I: The user language.
\emph{J. Symb. Comp.}, 24, 3/4:235-265, 1997.










\bibitem{DGS}
P.~Delsarte, J.-M.~Goethals, and J.J.~Seidel, Spherical codes and designs, 
{\sl Geom.~Dedicata} {\bf 6} (1977), 363--388. 















\bibitem{Jur} M.~J\"urgens, {\sl Nicht-Existenz und Konstruktion extremaler Gitter.} Dissertatio, TU Dortmund 2015. 

\bibitem{Kob} N.~Koblitz,
{\sl Introduction to Elliptic Curves and Modular Forms},
Springer-Verlag, Berlin/New York, 1984. 


\bibitem{KA}
S.D.~Kominers and Z.~Abel,
{Configurations of rank-$40r$ extremal even unimodular lattices ($r=1,2,3$)},
{\sl J.~Th\'eor.~Nombres Bordeaux}
{\bf 20} (2008), no.~2, 365--371.

\bibitem{K}
S.D.~Kominers, 
{Configurations of extremal even unimodular lattices},
{\sl Int.~J.~Number Theory}
{\bf 5} (2009), no.~3, 457--464.




\bibitem{O1}
M.~Ozeki,
{On even unimodular positive definite quadratic lattices of rank $32$},
{\sl Math.~Z.}
{\bf 191} (1986), 283--291.

\bibitem{O2}
M.~Ozeki,
{On the configurations of even unimodular lattices of rank $48$},
{\sl Arch.~Math.}
{\bf 46} (1986), 247--248.

\bibitem{O3}
M.~Ozeki,
{On the structure of even unimodular extremal lattices of rank $40$},
{\sl Rocky Mtn.~J.~Math.}
{\bf 19} (1989), 847--862.


\bibitem{Q}
H.-G.~Quebbemann, 
{Modular lattices in Euclidean spaces}, 
{\sl J.~Number Theory} 
{\bf 54} (1995), no.~2, 190--202. 

\bibitem{RS}
E.M.~Rains and N.J.A.~Sloane, 
{The shadow theory of modular and unimodular lattices}, 
{\sl J.~Number Theory} 
{\bf 73} (1998), no.~2, 359--389. 

\bibitem{SV}
R.~Scharlau and B.B.~Venkov, 
{The genus of the Barnes--Wall lattice}, 
{\sl Comment.~Math.~Helv.} 
{\bf 69} (1994), no.~2, 322--333. 










\bibitem{Venkov2}
B.B.~Venkov, R\'eseaux et designs sph\'eriques, (French) [Lattices and spherical designs] {\sl R\'eseaux euclidiens, designs sph\'eriques et formes modulaires}, 10--86, Monogr.~Enseign.~Math., 37, 
{\sl Enseignement Math., Geneva,} 2001.

\bibitem{Mathematica}
Wolfram Research, Inc., Mathematica, Version 11.2, Champaign, IL (2017).


\end{thebibliography}
\end{document}